\UseRawInputEncoding
\documentclass[12pt]{amsart}
\usepackage{}

\usepackage{cancel}
\usepackage{amsmath}
\usepackage{amsfonts}
\usepackage{amssymb}
\usepackage[all]{xy}           %xypic macro for latex2.09
\usepackage{multirow}
\usepackage{bigstrut}

\usepackage{bbding}
\usepackage{txfonts}
\usepackage{amscd}

\usepackage[shortlabels]{enumitem}
\usepackage{ifpdf}
\ifpdf
  \usepackage[colorlinks,final,backref=page,hyperindex]{hyperref}
\else
  \usepackage[colorlinks,final,backref=page,hyperindex,hypertex]{hyperref}
\fi
\usepackage{tikz}
\usepackage[active]{srcltx}

\makeatletter

\newtheorem{thm}{Theorem}[section]

\newtheorem{pro}[thm]{Proposition}
\newtheorem{ex}[thm]{Example}
\newtheorem{rmk}[thm]{Remark}
\newtheorem{defi}[thm]{Definition}

\newcommand {\emptycomment}[1]{}

\newcommand{\lon }{\,\rightarrow\,}
\newcommand{\be }{\begin{equation}}
\newcommand{\ee }{\end{equation}}

\newcommand{\g}{\mathfrak g}
\newcommand{\h}{\mathfrak h}

\newcommand{\huaB}{\mathcal{B}}%{{\mathcal{E}}}%{\mathcal{B}}

\newcommand{\huaF}{\mathcal{F}}

\newcommand{\huaD}{\mathcal{D}}

\newcommand{\huaH}{\mathcal{H}}

\newcommand{\GL}{\mathrm{GL}}

\newcommand{\dM}{\mathrm{d}}
\newcommand{\VE}{\mathsf{VE}}

\newcommand{\br}[1]{   [ \cdot,    \cdot  ]   }

\newcommand{\CE}{\mathsf{CE}}

\newcommand{\Hom}{\mathrm{Hom}}

\newcommand{\Der}{\mathrm{Der}}

\newcommand{\Ad}{\mathrm{Ad}}
\newcommand{\Aut}{\mathrm{Aut}}

\newcommand{\gl}{\mathfrak {gl}}

\newcommand{\ad}{\mathrm{ad}}

\newcommand{\Sym}{\mathsf{S}}

\begin{document}

\title{Deformations of crossed homomorphisms on Lie groups}

\author{Jun Jiang}
\address{Department of Mathematics, Jilin University, Changchun 130012, Jilin, China}
\email{junjiang@jlu.edu.cn}

%\date{\today}

\begin{abstract}
In this paper, we study deformations of crossed homomorphisms on Lie groups by means of the cohomology which controls them. Using the Moser type argument, we obtain several rigidity results of crossed homomorphisms on Lie groups. We further investigate the relationship between the cohomology of crossed homomorphisms on Lie groups and that on Lie algebras. Finally, we characterize the rigidity properties of all crossed homomorphisms with respect to the conjugation action on the connected and simply connected two-dimensional Lie group.
\end{abstract}

%\subjclass[2010]{17B10, 17B56, 17A42}


\maketitle

%\tableofcontents

%\setcounter{section}{0}

\allowdisplaybreaks

%\end{document}

\section{Introduction}
For a given algebraic or geometric structure, a central problem is that describe a neighborhood
of the structure within its moduli space. Studying deformations provides a powerful approach to this problem. Roughly speaking, a deformation of a structure is a small curve through the original structure in the moduli space. Cohomology theory offers essential tools for analyzing deformations: taking the differential of a deformation yields a
cocycle. Rigidity is a natural question need to be considered in studying deformations, and it can be characterized in terms of relevant cohomology groups. In differential geometry, the Moser argument is a classical method for investigating rigidity. Cohomological methods and Moser type arguments are similarly used to study deformations and rigidity phenomena of morphisms of given structures. For instance, cohomological approaches to deformations of Lie group homomorphisms appear in \cite{NR, CI}.

The purpose of this paper is to develop the cohomology theory governing deformations of crossed homomorphisms on Lie groups and use it to establish several
rigidity results. Our approach is an adaptation of the Moser argument to the context of deformations of crossed homomorphisms on Lie groups.

Crossed homomorphisms (also known as 1-cocycles or difference operators) on groups were first introduced by J.H.C. Whitehead in the study of relative homotopy groups\cite{Whi}. These operators found wide applications in non-abelian Galois cohomology\cite{Se} and the theory of Banach modules of locally compact groups\cite{Da}. In \cite{LYZ}, the authors constructed set-theoretical solutions of the Yang-Baxter equation using bijective crossed homomorphisms. Recently, crossed homomorphisms were used to study Hopf-Galois structures \cite{Tsang} and construct representations of mapping class groups of surfaces \cite{CS,Kasahara}. The concept of crossed homomorphisms on Lie algebras was introduced in \cite{Lue} during the study of non-abelian extensions of Lie algebras, and had been applied to study representations of Cartan type Lie algebras\cite{PSTZ}. By differentiation, crossed homomorphisms on Lie groups give rise to crossed homomorphisms on the corresponding Lie algebras. Conversely, crossed homomorphisms on Lie algebras can be integrated into crossed homomorphisms on the corresponding connected and simply connected Lie groups\cite{JS}.

Given a crossed homomorphism $D:G\lon H$ with respect to an action $\Theta:G\lon\Aut(H)$, i.e.,
$$
D(a\cdot_G b)=D(a)\cdot_H\Big(\Theta(a)D(b)\Big), \quad \forall a, b\in G.
$$
\begin{itemize}
\item We construct a representation $\Theta_D: G\lon\GL(\h)$ and define the cohomology of $D$ as the group cohomology of $G$ with coefficients in the representation $(\h,\Theta_D)$(Proposition \ref{reptd}). We establish the van Est isomorphism theorem relating the cohomology of crossed homomorphisms on Lie groups and the cohomology of crossed homomorphisms on Lie algebras(Theorem \ref{vanE}). We then introduce a deformation $D_s$ of $D$ and show that the smooth map $\hat{D}_s: G\lon\h$ defined by
$$
\hat{D}_{s}(a)=(R_{D_{s}(a)^{-1}})_{*D_{s}(a)}\frac{d}{dt}|_{t=0}D_{t+s}(a), \quad \forall a\in G.
$$
is a $1$-cocycle(Proposition \ref{eqcoD}).
\item A deformation $D_s$ of $D$ is said to be trivial on an interval $I=(-\varepsilon, \varepsilon)\subset\mathbb{R}$ if there exists a smooth map $\tau:I\lon H$ satisfying $$
\tau(0)=e_H, \quad D_t(a)=\tau(t)\cdot_HD(a)\cdot_H\Theta(a)\tau(t)^{-1}, \quad \forall a\in G,~t\in I.
$$
If the deformation $D_s$ is trivial on $I$, then the cohomology class $[\hat{D}_s]\in\huaH^{1}(D_s)$ vanishes smoothly for all $s\in I$ (Theorem \ref{trivial}). Conversely, if the cohomology classes $[\hat{D}_s]$ vanish smoothly for all $s\in\mathbb{R}$, then for each $a\in G$, there exists an interval $I_a=(-\varepsilon_a, \varepsilon_a)$ such that
$$
D_t(a)=\tau(t)\cdot_HD(a)\cdot_H\Theta(a)\tau(t)^{-1}, \quad \forall a\in G,~t\in I_a.
$$
The deformation $D_s$ is called to be pointwise trivial (since $I_a=(-\epsilon_a, \epsilon_a)$ depends on the point $a$)(Theorem \ref{trivial}). In particular, for a compact Lie group $G$, any deformation of $D: G\lon H$ is trivial(Proposition \ref{ctri}).
\item For the two-dimensional connected and simply connected Lie group $G$, there exists three classes of crossed homomorphisms with respect to the conjugation action (Proposition \ref{classfy}). By analyzing whether all deformations of crossed homomorphisms are trivial, we establish complete rigidity conditions for every crossed homomorphism with respect to the conjugation action on $G$(Theorems \ref{def1}, \ref{def3} and \ref{def2}).
\end{itemize}

The structure of the paper is the following. In Section $2$, we develop the cohomology theory for crossed homomorphisms on Lie groups. In Section $3$, we introduce deformations of these operators and prove several rigidity theorems. In Section $4$, we end this paper with characterizing the rigidity properties of all crossed homomorphisms with respect to the conjugation action on the connected and simply connected two-dimensional Lie group.

\vspace{2mm}
\noindent
{\bf Acknowledgements. }We give our warmest thanks to Yunhe Sheng for very helpful discussions. This research is supported by NSFC(12401076), China Postdoctoral Science Foundation (2023M741349).

\section{Cohomology of crossed homomorphisms on Lie groups}\label{22}
Let $(G, e_G, \cdot_G)$ and $(H, e_H, \cdot_H)$ be Lie groups. Denote by $(\h, [\cdot, \cdot]_\h)$ the Lie algebra of $(H, e_H, \cdot_H)$. We briefly recall some differential calculations in $G$ and $H$ that will be useful for our purposes.

Define a smooth map $m: G\times G\lon G$ by
\begin{equation*}
m(a, b)=a\cdot_G b, \quad \forall a, b\in G.
\end{equation*}
Denote by $L_a=m(a, \cdot): G\lon G$ and $R_b=m(\cdot, b):G\lon G$. For smooth curves $\gamma_1:\mathbb{R}\lon G$ and $\gamma_2: \mathbb{R}\lon G$, the following equation holds
\begin{eqnarray}
\label{defileib}\frac{d}{dt}|_{t=0}m(\gamma_1(t), \gamma_2(t))&=&\frac{d}{dt}|_{t=0}m(\gamma_1(0), \gamma_2(t))+\frac{d}{dt}|_{t=0}m(\gamma_1(t), \gamma_2(0))\\
\nonumber&=&{L_{\gamma_1(0)}}_{*\gamma_2(0)}\frac{d}{dt}|_{t=0}\gamma_2(t)+{R_{\gamma_2(0)}}_{*\gamma_1(0)}\frac{d}{dt}|_{t=0}\gamma_1(t),
\end{eqnarray}
where ${L_{\gamma_1(0)}}_{*\gamma_2(0)}$ and ${R_{\gamma_2(0)}}_{*\gamma_1(0)}$ are tangent maps of $L_{\gamma_1(0)}$ and $R_{\gamma_2(0)}$ at $\gamma_2(0)$ and $\gamma_1(0)$ respectively.

For the Lie group $(H, e_H, \cdot_H)$, denote by $\Ad: H\lon\Aut(\h)$ the adjoint action of $(H, e_H, \cdot_H)$ on $\h$. Then
\begin{equation}\label{defiad}
\Ad(h)u=\frac{d}{dt}|_{t=0}h\cdot_H\exp(tu)\cdot_H h^{-1}, \quad\forall h\in H, ~u\in\h,
\end{equation}
which means that $\Ad(h)u={(L_{h})}_{*h^{-1}}{(R_{h^{-1}})}_{*e_H}u={(R_{h^{-1}})}_{*h}{(L_{h})}_{*e_H}u$. Moreover, the linear map $\Ad_{*e_H}:\h\lon\Der(\h)$ is
\begin{equation*}
\Big(\Ad_{*e_H}(u)\Big)(v)=[u, v]_\h, \quad \forall u, v\in\h.
\end{equation*}

\emptycomment{
Let $\Theta: G\lon \Aut(H)$ be an action of $(G, e_G, \cdot_G)$ on $(H, e_H, \cdot_H)$. As $\Theta(a)\in\Aut(H)$ for all $a\in G$, then $\Theta(a)_{*e_H}:\h\lon\h$ is a Lie algebra isomorphism. By $\Theta(a\cdot_G b)=\Theta(a)\Theta(b)$, we have that $\Theta(a\cdot_G, b)_{*e_H}=\Theta(a)_{*e_H}\Theta(b)_{*e_H}$. Thus we obtain a Lie group homomorphism from $G$ to $\Aut(\h)$, which we denote by $\tilde{\Theta}: G\lon\Aut(\h)$. Then
\begin{equation}\label{defitheta}
\tilde{\Theta}(a)x=\frac{d}{dt}|_{t=0}\Theta(a)\gamma(t), \quad \forall a\in G,
\end{equation}
where $\gamma(0)=e_H$ and $\frac{d}{dt}|_{t=0}\gamma(t)=x$.
}
\subsection{Cohomology of crossed homomorphisms on Lie groups}
In this subsection, we introduce cohomology of crossed homomorphisms on Lie groups.
\begin{defi}{\rm(\cite{Whi})}
Let $(G, e_G, \cdot_G)$ and $(H, e_H, \cdot_H)$ be Lie groups and $\Theta: G\lon \Aut(H)$ be an action of $(G, e_G, \cdot_G)$ on $(H, e_H, \cdot_H)$. A smooth map $D:G\lon H$ is called a {\bf crossed homomorphism} with respect to the action $\Theta$ on Lie groups if it satisfies the following equation
\begin{equation}\label{deficro}
D(a\cdot_G b)=D(a)\cdot_H\Big(\Theta(a)D(b)\Big), \quad \forall a, ~b\in G.
\end{equation}

In particular, a smooth map $D:G\lon G$ is called a crossed homomorphism with respect to the conjugation action on $G$ if it satisfies the following equation
\begin{equation*}
D(a\cdot_G b)=D(a)\cdot_G a\cdot_G D(b)\cdot_G a^{-1}, \quad \forall a, ~b\in G.
\end{equation*}
\end{defi}

\begin{ex}\label{Ex1}
Let $(G, e_G, \cdot_G)$ be a Lie group. Then the inverse map $(\cdot)^{-1}: G\lon G$ is a crossed homomorphism with respect to the conjugation action.
\end{ex}

\begin{ex}
Let $(G, e_G, \cdot_G)$ and $(H, e_H, \cdot_H)$ be Lie groups and $\Theta: G\lon \Aut(H)$ be an action of $(G, e_G, \cdot_G)$ on $(H, e_H, \cdot_H)$. Let $\huaB: H\lon G$ be a relative Rota-Baxter operator, i.e. $\huaB:H\to G$ is a smooth map satisfying
$$
\huaB(h_1)\cdot_G\huaB(h_2)=\huaB(h_1\cdot_H\Theta(\huaB(h_1))h_2),\quad \forall h_1,h_2\in H.
$$  If $\huaB $ is invertible, then $\huaB^{-1}: G\lon H$ is a crossed homomorphism with respect to the action $\Theta$ on Lie groups.
\end{ex}

 Now we are ready to define a cohomology theory for crossed homomorphisms on Lie groups. First we recall a standard
 version of the smooth cohomology of a Lie group $G$ with coefficients in a representation $(V, \Pi)$ (see e.g. \cite{Wei}). An $n$-cochain is a smooth map $$\alpha_n:
\underbrace{G\times\cdots\times G}_n\lon V.$$  The set of $n$-cochains
forms an abelian group, which will be denoted by $C^{n}(G, V)$. The differential $\dM: C^{n}(G, V)\rightarrow C^{n+1}(G, V)$ is defined by
\begin{eqnarray*}
\dM(\alpha_n)(a_1,\cdots,a_n,a_{n+1})&=&\Pi(a_1)\alpha_n(a_2, \cdots, a_n, a_{n+1})+(-1)^{n+1}\alpha_{n}(a_1,\cdots,a_n)\\
&&+\sum_{i=1}^{n}(-1)^{i}\alpha_n(a_1,\cdots,a_{i-1}, a_i\cdot_{G}a_{i+1}, a_{i+2}, \cdots a_{n+1}).
\end{eqnarray*}

Let $D: G\lon H$ be a crossed homomorphism with respect to an action $\Theta$ on Lie groups. Denote by $\h$ the Lie algebra of the Lie group $(H, e_H, \cdot_H)$. Since $\Theta(a)\in\Aut(H)$ for each $a\in G$, the induced linear map $\Theta(a)_{*e_H}:\h\lon\h$ is a Lie algebra isomorphism. By $\Theta(a\cdot_G b)=\Theta(a)\Theta(b)$, we have that $\Theta(a\cdot_G b)_{*e_H}=\Theta(a)_{*e_H}\Theta(b)_{*e_H}$. Thus we obtain a Lie group homomorphism $\tilde{\Theta}: G\lon\Aut(\h)$ defined by
\begin{equation}\label{defitheta}
\tilde{\Theta}(a)x=\frac{d}{dt}|_{t=0}\Theta(a)\gamma(t), \quad \forall a\in G,
\end{equation}
where $\gamma(0)=e_H$ and $\frac{d}{dt}|_{t=0}\gamma(t)=x$. The crossed homomorphism $D$ with respect to an action $\Theta$ on Lie groups induces a smooth map $\Theta_D: G\lon\GL(\h)$ given by
\begin{equation}\label{defireptd}
\Theta_{D}(a)u=\Ad_{D(a)}\tilde{\Theta}(a)u, \quad\forall a\in G, ~u\in\h.
\end{equation}

\begin{pro}\label{reptd}
With the above notations, the map $\Theta_D: G\lon\GL(\h)$ is a representation of $(G, e_G, \cdot_G)$ on $\h$.
\end{pro}\begin{proof}
For any $a, b\in G, u\in h$, by \eqref{defiad}, \eqref{deficro} and \eqref{defitheta}, we obtain
\begin{eqnarray*}
\Ad_{D(a\cdot_G b)}\tilde{\Theta}(a\cdot_G b)u&=&\frac{d}{dt}|_{t=0}D(a\cdot_G b)\cdot_H\Big(\Theta(a\cdot_G b)\exp(tu)\Big)\cdot_HD(a\cdot_G b)^{-1}\\
&=&\frac{d}{dt}|_{t=0}D(a)\cdot_H\Theta(a)\Big(D(b)\cdot_H\Theta(b)\exp(tu)\cdot_HD(b)^{-1}\Big)\cdot_HD(a)^{-1}\\
&=&\Ad_{D(a)}\tilde{\Theta}(a)\Big(\frac{d}{dt}|_{t=0}D(b)\cdot_H\Theta(b)\exp(tu)\cdot_HD(b)^{-1}\Big)\\
&=&\Ad_{D(a)}\tilde{\Theta}(a)\Big(\Ad_{D(b)}\tilde{\Theta}(b)u\Big),
\end{eqnarray*}
which means that $\Theta_{D}: G\lon\GL(\h)$ defined by \eqref{defireptd} is a representation of $(G, e_G, \cdot_G)$ on $\h$.
\end{proof}

Define the space of $0$-cochains by $\h$. For $k\geq1$, define the space of $k$-cochains $$C^k(D)=\{\alpha_k:
\underbrace{G\times\cdots\times G}_k\lon \h|\alpha_k~~\text{is a smooth map}\}.$$
The differential $\dM^D: C^{k}(D)\rightarrow C^{k+1}(D)$ is defined by
\begin{eqnarray}
\label{deficob}&&\dM^D\alpha_k(a_1, \cdots, a_{k+1})\\
\nonumber&=&\Theta_D(a_1)\alpha_k(a_2, \cdots, a_{k+1})+(-1)^{k+1}\alpha_{k}(a_1,\cdots,a_k)\\
\nonumber&&+\sum_{i=1}^{k}(-1)^{i}\alpha_k(a_1,\cdots,a_{i-1}, a_i\cdot_{G}a_{i+1}, a_{i+2}, \cdots, a_{k+1}).
\end{eqnarray}
We know that the smooth map $\dM^D$ is the coboundary operator of the Lie group $(G, e_G, \cdot_G)$ with coefficients in the representation $(\h, \Theta_D)$, which means that $\dM^D\circ\dM^D=0$. Thus $(\oplus_{k=1}^{+\infty}C^k(D), \dM^D)$ is a cochain complex, i.e. $\dM^D\circ\dM^D=0$.

\begin{defi}
Let $D: G\lon H$ be a crossed homomorphism with respect to an action $\Theta$ on Lie groups. The cohomology of the cochain complex  $(C^*(D)=\oplus_{k=1}^{+\infty}C^k(D), \dM^D)$ is defined to be the {\bf cohomology for the crossed homomorphism $D$}.
\end{defi}
Denote by $Z^{k}(D)$ the set of $k$-cocycles and $\huaH^{k}(D)$ the $k$-th cohomology group.
\begin{rmk}
For a crossed homomorphism $D:G\lon G$ with respect to the the conjugation action on $G$, the corresponding cohomology groups coincide with those defined for the adjoint representation in \cite{JLS}.
\end{rmk}

\subsection{The van Est map}
In this subsection, we show that the cohomology of crossed homomorphisms on Lie groups is the global analogue of the cohomology of crossed homomorphisms on Lie algebras.

Let $(\g, [\cdot, \cdot]_\g)$ and $(\h, [\cdot, \cdot]_\h)$ be Lie algebras and $\theta:\g\lon\Der(\h)$ be an action of $\g$ on $\h$.
\begin{defi}{\rm(\cite{Lue})}
A linear map $\mathfrak{d}:\g\lon\h$ is called a {\bf crossed homomorphism} with respect to the action $\theta$ on Lie algebras if
\begin{equation*}
\frak{d}([x, y]_\g)=\theta(x)\frak{d}(y)-\theta(y)\frak{d}(x)+[\frak{d}(x), \frak{d}(y)]_{\h}, \quad \forall x,y\in\g.
\end{equation*}

In particular, a linear map $\frak{d}:\g\lon \g$ is called a crossed homomorphism on $\g$ if it satisfies the following equation
\begin{equation*}
\frak{d}([x, y]_\g)=[\frak{d}(x), y]_\g+[x, \frak{d}(y)]_\g+[\frak{d}(x), \frak{d}(y)]_\g, \quad \forall x, ~y\in \g.
\end{equation*}
\end{defi}

Let $(G, \cdot_G, e_G)$ and $(H, \cdot_H, e_H)$ be connected and simply connected Lie groups. Denote by $(\g, [\cdot, \cdot]_\g)$ and $(\h, [\cdot, \cdot]_\h)$ the Lie algebras of $(G, \cdot_G, e_G)$ and $(H, \cdot_H, e_H)$ respectively. Let $\Theta: G\lon\Aut(H)$ be an action of $G$ on $H$. Denote by $\theta:\g\lon\Der(\h)$ the action of $\g$ on $\h$ corresponding to $\Theta: G\lon\Aut(H)$. The correspondence between Lie group homomorphisms and their induced Lie algebra homomorphisms naturally extends to crossed homomorphisms.
\begin{thm}{\rm(\cite{GLS,JS})}
Let $D:G\lon H$ be a crossed homomorphism with respect to the action $\Theta$ on Lie groups. Then $\frak{d}=D_{*e_G}:\g\lon\h$ the tangent map of $D$ is a crossed homomorphism with respect to the action $\theta$ on Lie algebras.

Conversely, let $\frak{d}:\g\lon\h$ be a crossed homomorphism with respect to the action $\theta$ on Lie algebras. Then there exists a crossed homomorphism $D:G\lon H$ with respect to the action $\Theta$ on Lie groups such that $D_{*e_G}=\frak{d}$.
\end{thm}

At the end of this subsection, we establish the relationship between the cohomology of crossed homomorphisms on Lie groups and the cohomology of crossed homomorphism on Lie algebras. Recall the cohomology of crossed homomorphisms on Lie algebras as following, see \cite{JS, JST, PSTZ}for more details. Let $\frak{d}:\g\lon\h$ be a crossed homomorphism with respect to the action $\theta$ on Lie algebras. Define $\theta_{\frak{d}}:\g\to\gl(\h)$ by
\begin{equation}\label{eqreptheta}
\theta_{\frak{d}}(x)u=\theta(x)u+[\frak{d}(x), u]_\h, \quad \forall x\in\g, u\in\h.
\end{equation}
It was shown in \cite[Proposition 3.1]{Lue} that $\theta_{\frak{d}}$ is a representation of the Lie algebra $\g$ on $\h$. Denote by $C^{n}(\frak{d})=\Hom(\wedge^{n}\g, \h)$.

\begin{defi}{\rm(\cite{PSTZ})}
The cohomology of the cochain complex $(\oplus_{n=0}^{+\infty}C^n(\frak{d}),\dM^{\CE}_{\theta_{\frak{d}}})$ is called the {\bf cohomology of the crossed homomorphism $\frak{d}$}, where $\dM_{\theta_{\frak{d}}}^{\CE}:\Hom(\wedge^{k}\g, \h)\lon\Hom(\wedge^{k+1}\g, \h)$ is the Chevalley-Eilenberg coboundary operator of the Lie algebra $(\g, [\cdot,\cdot]_\g)$ with coefficients in the representation $(\h, \theta_{\frak{d}})$. Denote its $n$-th cohomology group by $\huaH^n(\frak{d})$.
\end{defi}

Let $D:G\lon H$ be a crossed homomorphism with respect to an action $\Theta$ on Lie groups. According to the Proposition \ref{reptd}, the map $\Theta_D: G\lon\GL(\h)$ given by
$$
\Theta_D(a)u=\Ad_{D(a)}\tilde{\Theta}(a)u, \quad\forall a\in G, u\in\h.
$$
is a representation of $G$ on $\h$. Denote by $\frak{d}:\g\lon\h$ the crossed homomorphism with respect to the action $\theta$ on Lie algebras corresponding to $D$.
\begin{pro}\label{thetadandT}
With the above notations, the tangent map of $\Theta_D:G\lon\GL(\h)$ is the linear map $\theta_{\frak{d}}:\g\lon\gl(\h)$.
\end{pro}
\begin{proof}
Denote by $\exp_G$ and $\exp_H$ are exponential maps of $(G, e_G, \cdot_G)$ and $(H, e_H, \cdot_H)$ respectively. For any $x\in\g$ and $u\in\h$, by \eqref{defireptd} and \eqref{eqreptheta}, we obtain
\begin{eqnarray*}
\frac{d}{ds}|_{s=0}\Theta_D(\exp_G(sx))u&=&\frac{d}{ds}|_{s=0}\Ad_{D(\exp_G(sx))}\tilde{\Theta}(\exp_G(sx))u\\
&=&\frac{d}{ds}|_{s=0}\Ad_{D(\exp_G(sx))}u+\frac{d}{ds}|_{s=0}\tilde{\Theta}(\exp_G(sx))u\\
&=&[\frak{d}(x), u]+\theta(x)u=\theta_{\frak{d}}(x)u.
\end{eqnarray*}
Thus the tangent map of $\Theta_D:G\lon\GL(\h)$ is $\theta_{\frak{d}}:\g\lon\gl(\h)$.
\end{proof}
Let $D:G\lon H$ be a crossed homomorphism with respect to the action $\Theta$ on Lie groups. Denote by $\frak{d}:\g\lon\h$ the crossed homomorphism with respect to the action $\theta$ on Lie algebras corresponding to $D$. Define $\VE:C^k(D)\lon C^{k}(\frak{d})$ by
\begin{eqnarray*}
&&\VE (F)(x_1,\cdots, x_{k})\\
\nonumber&=& \sum_{s\in \Sym(k)}(-1)^{|s|}\frac{d}{dt_{s(1)}}\cdots\frac{d}{dt_{s(k)}}\bigg|_{t_{s(1)}=t_{s(2)}=\cdots=t_{s(k)}=0}F(\exp_G(t_{s(1)}x_{s(1)}),\cdots,\exp_G(t_{s(k)}x_{s(k)})),
\end{eqnarray*}
for all $F\in C^k(D),~x_1,\cdots, x_{k}\in\g$, where $\exp_G$ is the exponential map of $(G, e_G, \cdot_G)$.
\begin{pro}
  With the above notations, $\VE$ is a cochain map, i.e. we have the following commutative diagram
\[
\small{ \xymatrix{
\cdots\longrightarrow C^{k-1}(D) \ar[d]^{\VE} \ar[r]^{\quad\dM^{D}}&
C^{k}(D) \ar[d]^{\VE} \ar[r]^{\quad\dM^{D}} & C^{k+1}(D) \ar[d]^{\VE} \ar[r]  & \cdots  \\
\cdots\longrightarrow C^{k-1}(\frak{d} ) \ar[r]^{\quad\dM^{\CE}_{\theta_{\frak{d}}}} & C^{k}(\frak{d}) \ar[r]^{\quad \dM^{\CE}_{\theta_{\frak{d}}}} &C^{k+1}(\frak{d})\ar[r]& \cdots.}
}
\]
Consequently, $\VE$ induces a homomorphism $\VE_*$ from the cohomology
group $\huaH^{k}(D)$ to $\huaH^{k}(\frak{d})$. The map $\VE$ is called the {\bf Van Est map}.
\end{pro}
\begin{proof}
By Proposition \ref{thetadandT}, the differentiation of the representation $\Theta_D:G\lon \GL(\h)$ of Lie group $(G, e_G, \cdot_G)$ is the representation $\theta_{\frak{d}}:\g\lon\gl(\h)$ of the Lie algebra $(\g, [\cdot, \cdot]_\g)$ . Since the cochains $C^k(D)$ and
$C^k(\frak{d})$ are in fact exactly those of Lie group $(G, e_G, \cdot_G)$ and the Lie algebra $(\g, [\cdot, \cdot]_\g)$ with representations $(\h, \Theta_D)$ and $(\h, \theta_{\frak{d}})$ respectively, the conclusion follows
from the classical argument for the cohomology of Lie groups and Lie
algebras.
\end{proof}

Just as the classical situation, under certain conditions, the cohomology group $\huaH^{k}(D)$ and $\huaH^{k}(\frak{d})$ are isomorphic.

\begin{thm}\label{vanE}
If the Lie group $(G, e_G, \cdot_G)$ is connected and its homotopy
groups are trivial in $1, 2,\cdots, n$, then for $1\leq k\leq n$, the
cohomology group $\huaH^{k}(D)$ of the crossed homomorphism $D$ on Lie groups is isomorphic to the cohomology
group $\huaH^{k}(\frak{d})$ of the crossed homomorphism $\frak{d}$ on Lie algebras.
\end{thm}
\begin{proof}
As for the previous proposition, the conclusion also follows from the classical argument for the cohomology of Lie groups and Lie algebras.
\end{proof}
\section{Deformations and rigidity of crossed homomorphisms on Lie groups}
In this section, we study deformations of crossed homomorphisms on Lie groups and obtain several rigidity results of them.
\subsection{Deformations of crossed homomorphisms on Lie groups}
Let $(G, e_G, \cdot_G)$ and $(H, e_H, \cdot_H)$ be Lie groups and $\Theta: G\lon \Aut(H)$ be an action of $(G, e_G, \cdot_G)$ on $(H, e_H, \cdot_H)$. Let $D: G\lon H$ be a crossed homomorphism with respect to the action $\Theta$ on Lie groups.

\begin{defi}
A {\bf deformation} of $D$ is a smooth map $\huaD(\cdot, \cdot): G\times \mathbb{R}\lon H$ such that $\huaD(\cdot, 0)=D$ and $\huaD(\cdot, t):G\lon H$ are crossed homomorphisms with respect to the action $\Theta$ on Lie groups for all $t\in\mathbb{R}$.
\end{defi}
\emptycomment{
\begin{rmk}\label{excl}
For a crossed homomorphism $D: G\lon H$ of Lie groups with respect to the action $\Theta$, we known that $\huaD(a, t)=D(a)$ always is a deformation of $D$. In the sequel, the deformations we mentioned always exclude this deformation.
\end{rmk}
}
\begin{defi}
Let $\huaD$ and $\huaD'$ be two deformations of $D$. If there exists an interval $I=(-\varepsilon, \varepsilon)\subset\mathbb{R}$ and a smooth map $\tau:I\lon H$ such that $\tau(0)=e_H$ and
\begin{equation*}
\huaD(a, t)=\tau(t)\cdot_H\huaD'(a, t)\cdot_H \Theta(a)\tau(t)^{-1},
\end{equation*}
for all $t\in I$, then $\huaD$ and $\huaD'$ are called to be {\bf equivalent} on $I$.
\end{defi}

Let $\huaD(\cdot, \cdot): G\times \mathbb{R}\lon H$ be a deformation of $D$. Denote by $D_s=\huaD(\cdot, s)$ for any $s\in\mathbb{R}$. Then $D_s: G\lon H$ is a crossed homomorphism with respect to the action $\Theta$ on Lie groups. Define a smooth map $\hat{D}_{s}: G\lon\h$ by
\begin{equation}\label{eqcoD}
\hat{D}_{s}(a)=(R_{D_{s}(a)^{-1}})_{*D_{s}(a)}\frac{d}{dt}|_{t=0}D_{t+s}(a), \quad \forall a\in G.
\end{equation}
\begin{pro}\label{cocycle}
With the above notations, $\hat{D}_{s}$ is a $1$-cocycle of the crossed homomorphism $D_s$, i.e. $\hat{D}_{s}\in Z^{1}(D_s)$. Moreover, if $\huaD$ and $\huaD'$ are equivalent deformations of $D$, then $[\hat{D'}_0]=[\hat{D}_0]\in\huaH^{1}(D)$.
\end{pro}
\begin{proof}
Since $\huaD(\cdot, \cdot): G\times \mathbb{R}\lon H$ is a deformation of $D$, then for each $s\in\mathbb{R}$,
$$
D_{s}(a\cdot_G b)=D_{s}(a)\cdot_H\Big(\Theta(a)D_s(b)\Big), \quad\forall a, b\in G.
$$
By \eqref{defileib}, \eqref{defiad} and \eqref{defitheta}, we have
\begin{eqnarray*}
&&\frac{d}{dt}|_{t=0}D_{t+s}(a\cdot_G b)\\
&=&\frac{d}{dt}|_{t=0}D_{t+s}(a)\cdot_H\Big(\Theta(a)D_{t+s}(b)\Big)\\
&=&({R_{\Theta(a)D_{s}(b)}})_{*D_s(a)}\frac{d}{dt}|_{t=0}D_{t+s}(a)+({L_{D_{s}(a)}})_{*\Theta(a)D_s(b)}\frac{d}{dt}|_{t=0}\Theta(a)D_{t+s}(b)\\
&=&({R_{\Theta(a)D_{s}(b)}})_{*D_s(a)}({R_{D_{s}(a)}})_{*e_H}\hat{D}_{s}(a)+({L_{D_{s}(a)}})_{*\Theta(a)D_s(b)}\frac{d}{dt}|_{t=0}\Big(\Theta(a)(D_{t+s}(b)\cdot_HD_{s}(b)^{-1})\cdot_H\Theta(a)D_{s}(b)\Big)\\
&=&({R_{\Theta(a)D_{s}(b)}})_{*D_s(a)}({R_{D_{s}(a)}})_{*e_H}\hat{D}_{s}(a)+({L_{D_{s}(a)}})_{*\Theta(a)D_s(b)}({R_{\Theta(a)D_s(b)}})_{*e_H}\frac{d}{dt}|_{t=0}\Big(\Theta(a)(D_{t+s}(b)\cdot_HD_{s}(b)^{-1})\Big)\\
%&=&{R_{\Theta(a)D_{s}(b)}}_{*}{R_{D_{s}(a)}}_{*}\hat{D}_{s}(a)+{L_{D_{s}(a)}}_{*}{R_{\Theta(a)D_s(b)}}_{*}\frac{d}{dt}|_{t=0}\Big(\Theta(a)(D_{t+s}(b)\cdot_HD_{s}(b)^{-1})\Big)\\
&=&({R_{\Theta(a)D_{s}(b)}})_{*D_s(a)}({R_{D_{s}(a)}})_{*e_H}\hat{D}_{s}(a)+({L_{D_{s}(a)}})_{*\Theta(a)D_s(b)}({R_{\Theta(a)D_s(b)}})_{*e_H}\tilde{\Theta}(a)\hat{D}_s(b)\\
&=&({R_{D_s(a)\cdot_H(\Theta(a)D_{s}(b))}})_{*e_H}\hat{D}_{s}(a)+({R_{\Theta(a)D_s(b)}}_{*})_{*D_s(a)}({L_{D_{s}(a)}})_{*e_H}\tilde{\Theta}(a)\hat{D}_s(b)\\
&=&({R_{D_s(a)\cdot_H(\Theta(a)D_{s}(b))}})_{*e_H}\hat{D}_{s}(a)+({R_{D_s(a)\cdot_H(\Theta(a)D_{s}(b))}})_{*e_H}\Ad_{D_{s}(a)}\tilde{\Theta}(a)\hat{D}_s(b)\\
&=&({R_{D_s(a\cdot_G b)}})_{*e_H}\hat{D}_{s}(a)+({R_{D_s(a\cdot_G b)}})_{*e_H}\Ad_{D_{s}(a)}\tilde{\Theta}(a)\hat{D}_s(b),
\end{eqnarray*}
thus
$$
\hat{D}_s(a\cdot_G b)=\hat{D}_s(a)+\Ad_{D_s(a)}\tilde{\Theta}(a)\hat{D}_s(b),
$$
which means that $\hat{D}_s$ is a $1$-cocycle of crossed homomorphism $D_s$, i.e. $\hat{D}_s\in Z^{1}(D_s)$.

Since two deformations of $D$ are equivalent, there exists a small enough real number $\varepsilon>0$ and a smooth map $\tau: (-\varepsilon, \varepsilon)\lon H$ such that
$$\tau(0)=e_H,\quad\huaD(a, t)=\tau(t)\cdot_H\huaD'(a, t)\cdot_H\Theta(a)\tau(t)^{-1}, \quad\forall t\in(-\varepsilon, \varepsilon), ~~a\in G.$$
Denote by $\frac{d}{dt}|_{t=0}\tau(t)=w\in\h$. By \eqref{defileib}, \eqref{defiad} and \eqref{defitheta}, we have
\begin{eqnarray*}
&&\hat{D}'_{0}(a)-\hat{D}_0(a)\\
&=&(R_{D(a)^{-1}})_{*D(a)}\frac{d}{d}|_{t=0}D'_{t}(a)-(R_{D(a)^{-1}})_{*D(a)}\frac{d}{dt}|_{t=0}D_t(a)\\
&=&(R_{D(a)^{-1}})_{*D(a)}\frac{d}{d}|_{t=0}D'_{t}(a)-(R_{D(a)^{-1}})_{*D(a)}\frac{d}{dt}|_{t=0}\tau(t)\cdot_HD'_{t}(a)\cdot_H\Theta(a)\tau(t)^{-1}\\
&=&(R_{D(a)^{-1}})_{*D(a)}\frac{d}{d}|_{t=0}D'_{t}(a)-(R_{D(a)^{-1}})_{*D(a)}\frac{d}{dt}|_{t=0}D'_{t}(a)\cdot_H\Theta(a)\tau(t)^{-1}-(R_{D(a)^{-1}})_{*D(a)}(R_{D(a)})_{*e_H}w\\
&=&(R_{D(a)^{-1}})_{*D(a)}\frac{d}{d}|_{t=0}D'_{t}(a)+(R_{D(a)^{-1}})_{*D(a)}(L_{D(a)})_{*e_H}\tilde{\Theta}(a)w-(R_{D(a)^{-1}})_{*D(a)}\frac{d}{dt}|_{t=0}D'_{t}(a)-w\\
&=&\Ad_{D(a)}\tilde{\Theta}(a)w-w=\dM^{D}w(a).
\end{eqnarray*}
Therefore $[\hat{D'}_0]=[\hat{D}_0]\in\huaH^{1}(D)$.
\end{proof}

\subsection{Rigidity of crossed homomorphisms on Lie groups}\label{secri}
In this subsection, $(G, e_G, \cdot_G)$ and $(H, e_H, \cdot_H)$ are Lie groups and $\Theta: G\lon \Aut(H)$ is an action of $G$ on $H$. Denote by $\h$ the Lie algebra of $(H, e_H, \cdot_H)$. Let $D: G\lon H$ be a crossed homomorphism with respect to the action $\Theta$ on Lie groups, and let $\huaD: G\times\mathbb{R}\lon H$ be a deformation of $D$. Denote by $D_s=\huaD(\cdot, s)$. According to the Proposition \ref{cocycle}, $\hat{D}_s$ given by
$$
\hat{D}_s={R_{D_{s}(a)^{-1}}}_{*D_s(a)}\frac{d}{dt}|_{t=0}D_{t+s}(a), \quad \forall s\in\mathbb{R}, ~~a\in G,
$$
is a $1$-cocycle, i.e. $\hat{D}_s\in Z^{1}(D_s)$.
\begin{defi}
A deformation $\huaD: G\times\mathbb{R}\lon H$ of $D$ is {\bf trivial} on $I=(-\varepsilon, \varepsilon)\subset\mathbb{R}$ if there exists a smooth map $\tau:I\lon H$ such that $$
\tau(0)=e_H, \quad D_t(a)=\tau(t)\cdot_HD(a)\cdot_H\Theta(a)\tau(t)^{-1}, \quad \forall a\in G,~t\in I,
$$
i.e. $D_t: G\lon H$ is equivalent to $D$ on $I$. The crossed homomorphism $D$ is said to be {\bf rigid} if all deformations of $D$ are trivial.

If for each $a\in G$, there exists an interval $I_a=(-\varepsilon_a, \varepsilon_a)$ such that
$$
D_t(a)=\tau(t)\cdot_HD(a)\cdot_H\Theta(a)\tau(t)^{-1}, \quad \forall a\in G,~t\in I_a,
$$
The deformation $\huaD: G\times\mathbb{R}\lon H$ is called to be {\bf pointwise trivial}. The crossed homomorphism $D$ is said to be {\bf weakly rigid} if all deformations of $D$ are pointwise trivial.
\end{defi}
\emptycomment{
According to the Example \ref{Ex1}, the inverse map $(\cdot)^{-1}: G\lon G$ is a crossed homomorphism of $G$ with respect to the conjugation action. Suppose that there exist a deformation $\huaD: G\times\mathbb{R}\lon G$ of $(\cdot)^{-1}$ is pointwise trivial, then for any $a\in G$, there exist an interval $I_a=(-\varepsilon_a, \varepsilon_a)$ such that
$$
D_t(a)=\tau(t)\cdot_G a^{-1}\cdot_G a\cdot_G\tau(t)^{-1}\cdot_G a^{-1}=a^{-1}, \quad \forall a\in G,~t\in I_a.
$$
According to the Remark \ref{excl}, we have the following proposition.
\begin{pro}
Any deformation of $(\cdot)^{-1}$ is non pointwise trivial.
\end{pro}
}
\begin{defi}
Let $\huaD:G\times\mathbb{R}\lon H$ be a deformation of $D$ and $c_{s}\in Z^{1}(D_s)$ be cocycles. We say that the cohomology classes $[c_{s}]\in\huaH^{1}(D_s)$ {\bf vanish smoothly} on $\mathbb{R}$ if there exists a smooth map $\kappa:\mathbb{R}\lon\h$ such that $\dM^{D_s}(\kappa(s))=c_s$ for all $s\in \mathbb{R}$.
\end{defi}

\begin{thm}\label{trivial}
Let $\huaD:G\times\mathbb{R}\lon H$ be a deformation of $D$. If $\huaD$ is trivial on $\mathbb{R}$, then the cohomology classes $[\hat{D}_s]\in\huaH^{1}(D_s)$ vanish smoothly on $\mathbb{R}$.

Conversely, if the cohomology classes $[\hat{D}_s]$ vanish smoothly on $\mathbb{R}$, then $\huaD$ is a pointwise trivial deformation of $D$.
\end{thm}
\begin{proof}
Assume that the deformation $\huaD:G\times\mathbb{R}\lon H$ of $D$ is trivial on $\mathbb{R}$, then there exists a smooth map $\tau:\mathbb{R}\lon H$ such that
\begin{equation*}
D_s(a)=\tau(s)\cdot_{H}D(a)\cdot_H\Big(\Theta(a)\tau(s)^{-1}\Big), \quad\forall s\in \mathbb{R}, ~~a\in G.
\end{equation*}
Define a smooth map $\kappa: \mathbb{R}\lon\h$ by
$$
\kappa(s)=-\frac{d}{dt}|_{t=0}\tau(t+s)\cdot_H\tau(s)^{-1}, \quad\forall s\in I.
$$
For any $a\in G$ and $s\in\mathbb{R}$, we have
\begin{eqnarray*}
&&\frac{d}{dt}|_{t=0}D_{s+t}(a)\\
&=&\frac{d}{dt}|_{t=0}\Big(\tau(s+t)\cdot_{H}D(a)\cdot_H\Theta(a)\tau(s+t)^{-1}\Big)\\
&=&\frac{d}{dt}|_{t=0}\Big((\tau(s+t)\cdot_H\tau(s)^{-1})\cdot_{H}(\tau(s)\cdot_HD(a)\cdot_H\Theta(a)\tau(s)^{-1})\cdot_H\Theta(a)(\tau(s+t)\cdot_H\tau(s)^{-1})^{-1}\Big)\\
&=&\frac{d}{dt}|_{t=0}\Big((\tau(s+t)\cdot_H\tau(s)^{-1})\cdot_{H}D_s(a)\cdot_H\Theta(a)(\tau(s+t)\cdot_H\tau(s)^{-1})^{-1}\Big)\\
&=&{L_{D_{s}(a)}}_{*e_H}\tilde{\Theta}(a)\kappa(s)-{R_{D_{s}(a)}}_{*e_H}\kappa(s),
\end{eqnarray*}
which implies that
\begin{eqnarray*}
\hat{D}_{s}(a)&=&{R_{D_s(a)^{-1}}}_{*D_s(a)}\frac{d}{dt}|_{t=0}D_{s+t}(a)={R_{D_s(a)^{-1}}}_{*D_s(a)}\Big({L_{D_{s}(a)}}_{*e_H}\tilde{\Theta}(a)\kappa(s)-{R_{D_{s}(a)}}_{*e_H}\kappa(s)\Big)\\
&=&\Ad_{D_s(a)}\tilde{\Theta}(a)\kappa(s)-\kappa(s).
\end{eqnarray*}
Thus $\hat{D}_s=\dM^{D_s}\kappa(s)$, which means that the cohomology classes $[\hat{D}_s]\in\huaH^{1}(D_s)$ vanish smoothly on $\mathbb{R}$.

Conversely, assume that the cohomology classes $[\hat{D}_s]$ associated with $\huaD$ vanish smoothly on $\mathbb{R}$, then there exists a smooth map $\alpha: \mathbb{R}\lon\h$ such that
$$
\hat{D}_{s}=\dM^{D_s}\alpha(s) \quad\forall s\in\mathbb{R}.
$$
Define a time-dependent vector field on $H$ by
$$
\alpha^{R}_{s}(h)=-{R_{h}}_{*e_H}\alpha(s), \quad\forall s\in\mathbb{R}, ~h\in H.
$$
Then there exists $\varepsilon>0$ such that the flow $\psi_t(e_H)$ is defined for all $t\in I=(-\varepsilon, \varepsilon)$. Define a smooth map $\tau: I\lon H$ by
$$
\tau(s)=\psi_{s}(e_H), \quad \forall s\in I.
$$
Then
\begin{equation*}
\frac{d}{ds}\tau(s)=\frac{d}{dt}|_{t=0}\tau(s+t)=\frac{d}{dt}|_{t=0}\psi_{s+t}(e_H)=\alpha^{R}_{s}(\psi_s(e_H))=-{R_{\tau(s)}}_{*e_H}\alpha(s).
\end{equation*}
Define a smooth map $\huaF:G\times I\lon H$ by
\begin{equation*}
\huaF(a, s)=\tau(s)\cdot_HD(a)\cdot_H\Theta(a)\tau(s)^{-1},\quad\forall s\in I,~a\in G.
\end{equation*}
Denote by $F_s=\huaF(\cdot, s)$ and $\hat{F}_{s}(a)={R_{F_s(a)^{-1}}}_{*e_H}\frac{d}{dt}|_{t=0}F_{t+s}(a)$ for all $a\in G$. Then we have that
\begin{eqnarray*}
F_s(a\cdot_G b)&=&\tau(s)\cdot_HD(a\cdot_G b)\cdot_H\Theta(a\cdot_G b)\tau(s)^{-1}\\
&=&\tau(s)\cdot_HD(a)\cdot_H\Theta(a)\Big(D(b)\cdot_H\Theta(b)\tau(s)^{-1}\Big)\\
&=&\tau(s)\cdot_HD(a)\cdot_H\Theta(a)\tau(s)^{-1}\cdot_H\Theta(a)\Big(\tau(s)\cdot_HD(b)\cdot_H\Theta(b)\tau(s)^{-1}\Big)\\
&=&F_s(a)\cdot_H\Theta(a)F_s(b),
\end{eqnarray*}
which means that $F_s: G\lon H$ is a crossed homomorphism with respect to the action $\Theta$ on Lie groups. Moreover, $\huaF(a, 0)=D(a)$, it follows that $\huaF: G\times I\lon G$ is a deformation of $D$. According to Proposition \ref{cocycle}, we obtain $[\hat{F}_s]\in\huaH^{1}(F_s)$ for all $s\in I$.

For a fix $a\in G$, define a time-dependent field on $H$ by
\begin{equation*}
X_{s}(h)={L_{h}}_{*e_H}\tilde{\Theta}(a)\alpha(s)-{R_{h}}_{*e_H}\alpha(s), \quad \forall s\in I, ~h\in H.
\end{equation*}
Then we have
\begin{eqnarray*}
X_{s}(D_s(a))&=&{R_{D_s(a)}}_{*e_H}\Big(\Ad_{D_s(a)}\tilde{\Theta}(a)\alpha(s)-\alpha(s)\Big)\\
&=&{R_{D_s(a)}}_{*e_H}\Big(\dM^{D_s}\alpha(s)(a)\Big)={R_{D_s(a)}}_{*e_H}\hat{D}_{s}(a)=\frac{d}{dt}|_{t=0}D_{t+s}(a)=\frac{d}{dt}|_{t=s}D_{t}(a),
\end{eqnarray*}
and
\begin{eqnarray*}
&&\frac{d}{dt}|_{t=s}F_{t}(a)\\
&=&\frac{d}{dt}|_{t=0}F_{t+s}(a)\\
&=&\frac{d}{dt}|_{t=0}\Big(\tau(s+t)\cdot_H D(a)\cdot_H\Theta(a)\tau(s+t)^{-1}\Big)\\
&=&{R_{(D(a)\cdot_H\Theta(a)\tau(s))}}_{*\tau(s)}\frac{d}{dt}|_{t=0}\tau(s+t)+{L_{(\tau(s)\cdot_HD(a)\cdot_H\Theta(a)\tau(s)^{-1})}}_{*e_H}\frac{d}{dt}|_{t=0}\Big(\Theta(a)(\tau(s+t)\cdot_H\tau(s)^{-1})\Big)^{-1}\\
&=&-{R_{(D(a)\cdot_H\Theta(a)\tau(s))}}_{*\tau(s)}{R_{\tau(s)}}_{*e_H}\alpha(s)+{L_{(\tau(s)\cdot_HD(a)\cdot_H\Theta(a)\tau(s)^{-1})}}_{*e_H}\tilde{\Theta}(a)\alpha(s)\\
&=&{L_{F_s(a)}}_{*e_H}\tilde{\Theta}(a)\alpha(s)-{R_{F_s(a)}}_{*e_H}\alpha(s)\\
&=&X_{s}(F_{s}(a)).
\end{eqnarray*}
We have that $D_s(a)$ and $F_s(a)$ are integral curves of the time-dependent vector field $X_{s}$ both passing through $D(a)$ at time $s=0$. Thus there exists $I'=(-\varepsilon', \varepsilon')\subset I$ such that
$$
\tau(s)\cdot_HD(a)\cdot_H\Theta(a)\tau(s)^{-1}=D_s(a), \quad\forall s\in I'.
$$
Therefore, for any $a\in G$, we can find an interval(which depends on $a$) $I_a\subset I$ such that $$D_s(a)=\tau(s)\cdot_H D(a)\cdot_H\Theta(a)\tau(s)^{-1} \quad \forall s\in I_a,$$
which means that $\huaD$ is a pointwise trivial deformation of $D$.
\end{proof}
A close relationship exists between Rota-Baxter operators on Lie groups and crossed homomorphisms on Lie groups. However, the relationship between deformations of Rota-Baxter operators on Lie groups and their cohomology theory defined in \cite{JSZ} remains unexplored. This connection merits further consideration.

\begin{pro}\label{ctri}
Let $G$ be a compact Lie group and $D: G\lon H$ be a crossed homomorphism with respect to an action $\Theta$ on Lie groups. Then any deformation $\huaD: G\times\mathbb{R}\lon H$ of $D$ is trivial, hence all crossed homomorphisms with respect to an action $\Theta$ on $G$ are rigid.
\end{pro}
\begin{proof}
Consider the Haar measure on $G$, i.e. a measure such that
$$
\int_Gdb=1,\quad  \int_Gf(ab)db=\int_Gf(b)db, \quad \text{for all} ~f\in C^{\infty}(G) ~\text{and for any fix} ~a\in G.
$$
Let $\huaD: \mathbb{R}\times G\lon H$ be a deformation of $D$. By Proposition \ref{cocycle}, $\hat{D}_s\in Z^{1}(D_s)$, that is
$$
\dM^{D_s}\hat{D}_s(a, b)=\Theta_{D_s}(a)\hat{D}_s(b)+\hat{D}_s(a)-\hat{D}_s(a\cdot_G b)=0.
$$
Denote by $\kappa(s)=-\int_G\hat{D}_s(b)db\in\h$. By integration one has
\begin{eqnarray*}
\hat{D}_s(a)&=&\int_G\hat{D}_s(a)db\\
&=&-\int_G\Theta_{D_s}(a)\hat{D}_s(b)db+\int_G\hat{D}_s(a\cdot_G b)db\\
&=&-\Theta_{D_s}(a)\int_G\hat{D}_s(b)db+\int_G\hat{D}_s(b)db\\
&=&\Theta_{D_s}\kappa(s)-\kappa(s)\\
&=&\dM^{D_s}(\kappa(s))(a).
\end{eqnarray*}
Thus $[\hat{D}_s]\in\huaH^{1}(D_s)$ vanish smoothly on $\mathbb{R}$. By Theorem \ref{trivial}, any deformation $\huaD: G\times\mathbb{R}\lon H$ of $D$ is trivial.
\end{proof}

\section{Deformations of crossed homomorphisms on the two-dimensional connected and simply connected Lie group}
Consider the $2$-dimensional real Lie algebra
$$\Big(\g=\text{span}\{e_1=\left(
\begin{array}{cc}
1 & 0 \\
0 & -1 \\
\end{array}
\right), e_2=\left(
\begin{array}{cc}
0 & 1 \\
0 & 0 \\
\end{array}
\right)
\}, ~~~~[\cdot, \cdot]\Big),$$ where $$[e_1, e_2]=e_1e_2-e_2e_1=2e_2.$$
Then $G=\{\left(                                                                                                                             \begin{array}{cc}                                                                                                                                 a & b \\                                                                                                                                 0 & \frac{1}{a} \\                                                                                                                               \end{array}                                                                                                                             \right)| a>0, b\in\mathbb{R}
\}$ is the connected and simply connected Lie group integrating the Lie algebra $\g$.

Let $D:G\lon G$ be a crossed homomorphism with respect to the conjugation action on $G$, i.e.,
$$
D(a\cdot_G b)=D(a)\cdot_G a\cdot_G D(b)\cdot_G a^{-1}, \quad \forall a,~b\in G.
$$
In the following, we will characterize the rigidity properties of all crossed homomorphisms with respect to the conjugation action on $G$.
\begin{pro}\label{classfy}
There are three classes of crossed homomorphisms with respect to the conjugation action on $G$ as the following:
\begin{eqnarray*}
D(\left(
   \begin{array}{cc}
     a & b \\
     0 & a^{-1} \\
   \end{array}
 \right))&=&\left(
           \begin{array}{cc}
             1 & \mu ab+\frac{\lambda}{2}(a^{2}-1) \\
             0 & 1 \\
           \end{array}
         \right),\quad \forall \mu, \lambda\in\mathbb{R},\\
\tilde{D}(\left(
  \begin{array}{cc}
    a & b \\
    0 & a^{-1} \\
  \end{array}
\right))&=&\left(
          \begin{array}{cc}
            a^{-1} & -b+qa\ln a \\
            0 & a \\
          \end{array}
        \right),\quad \forall q\in\mathbb{R},
\end{eqnarray*}
and
\begin{equation*}
\bar{D}(\left(
   \begin{array}{cc}
     a & b \\
     0 & a^{-1} \\
   \end{array}
 \right))=\left(
           \begin{array}{cc}
             a^{p} &  q\frac{a^{2+p}-a^{-p}}{2(p+1)}-ba^{1+p}\\
             0 & a^{-p} \\
           \end{array}
         \right),\quad p, q\in\mathbb{R}~~\text{and}~~p\neq 0, -1.
\end{equation*}
\end{pro}
\begin{proof}
A linear map $\frak{d}:\g\lon\g$ given by
$$
\frak{d}(e_1, e_2)=(e_1, e_2)
\left(
\begin{array}{cc}
d_{11} & d_{12} \\
d_{21} & d_{22} \\
\end{array}
\right)
$$
is a crossed homomorphism with respect to $\ad$ on $\g$ if and only if
$$
\frak{d}([e_1, e_2])=[\frak{d}(e_1), e_2]+[e_1, \frak{d}(e_2)]+[\frak{d}(e_1), \frak{d}(e_2)].
$$
By straightforward computation, there are the following two classes of crossed homomorphisms on $\g$:
$$
\frak{d}=
\left(
\begin{array}{cc}
0 & 0 \\
\lambda & \mu \\
\end{array}
\right), \quad \text{or} \quad
\frak{d}=
\left(
\begin{array}{cc}
p & 0 \\
q & -1 \\
\end{array}
\right),
\quad p\neq0.
$$
Thus, we have
$$
\frak{d}(\left(
     \begin{array}{cc}
       x & y \\
       0 & -x \\
     \end{array}
   \right))=\left(
             \begin{array}{cc}
               0 & \lambda x+\mu y \\
               0 & 0 \\
             \end{array}
           \right)
, \quad\text{and}\quad
\frak{d}(\left(
     \begin{array}{cc}
       x & y \\
       0 & -x \\
     \end{array}
   \right))=\left(
             \begin{array}{cc}
               px & qx-y \\
               0 & -px \\
             \end{array}
           \right),\quad p\neq 0.
$$
Indeed, every crossed homomorphism on $\g$ can be integrated into a unique crossed homomorphism on $G$ according to the \cite[Theorem 5.5]{JS}. By straightforward computation, the following smooth maps
\begin{eqnarray*}
D(\left(
   \begin{array}{cc}
     a & b \\
     0 & a^{-1} \\
   \end{array}
 \right))&=&\left(
           \begin{array}{cc}
             1 & \mu ab+\frac{\lambda}{2}(a^{2}-1) \\
             0 & 1 \\
           \end{array}
         \right),\quad \forall \mu, ~~\lambda\in\mathbb{R},\\
\tilde{D}(\left(
  \begin{array}{cc}
    a & b \\
    0 & a^{-1} \\
  \end{array}
\right))&=&\left(
          \begin{array}{cc}
            a^{-1} & -b+qa\ln a \\
            0 & a \\
          \end{array}
        \right),\quad \forall q\in\mathbb{R},
\end{eqnarray*}
and
\begin{equation*}
\bar{D}(\left(
   \begin{array}{cc}
     a & b \\
     0 & a^{-1} \\
   \end{array}
 \right))=\left(
           \begin{array}{cc}
             a^{p} &  q\frac{a^{2+p}-a^{-p}}{2(p+1)}-ba^{1+p}\\
             0 & a^{-p} \\
           \end{array}
         \right),\quad p, q\in\mathbb{R}~~\text{and}~~p\neq 0, -1,
\end{equation*}
are all crossed homomorphisms with respect to the conjugation action. Moreover
\begin{eqnarray*}
\frac{d}{dt}|_{t=0}D(\left(
   \begin{array}{cc}
     e^{tx} & ty \\
     0 & e^{-tx} \\
   \end{array}
 \right))&=&\left(
           \begin{array}{cc}
             0 & \mu y+\lambda x \\
             0 & 0 \\
           \end{array}
         \right),\quad \forall \mu, \lambda\in\mathbb{R},\\
\frac{d}{dt}|_{t=0}\tilde{D}(\left(
  \begin{array}{cc}
    e^{tx} & ty \\
    0 & e^{-tx} \\
  \end{array}
\right))&=&\left(
          \begin{array}{cc}
            -x & -y+qx \\
            0 & x \\
          \end{array}
        \right),\quad \forall q\in\mathbb{R},
\end{eqnarray*}
and
\begin{equation*}
\frac{d}{dt}\bar{D}(\left(
   \begin{array}{cc}
     e^{tx} & ty \\
     0 & e^{-tx} \\
   \end{array}
 \right))=\left(
           \begin{array}{cc}
             px &  qx-y\\
             0 & -px \\
           \end{array}
         \right),\quad p, q\in\mathbb{R}~~\text{and}~~p\neq 0, -1.
\end{equation*}
Therefore, there are three classes of crossed homomorphisms with respect to the conjugation action on $G$.
\end{proof}
Define three sets as following:
\begin{eqnarray*}
 \Gamma_1&=&\Big\{D: G\lon G~~\big|~~D(\left(
  \begin{array}{cc}
    a & b \\
    0 & a^{-1} \\
  \end{array}
\right))=\left(
          \begin{array}{cc}
            a^{-1} & -b+qa\ln a \\
            0 & a \\
          \end{array}
        \right),~~~~\forall q\in\mathbb{R} \Big\},\\
\Gamma_2&=&\Big\{D: G\lon G~~\big|~~~~D(\left(
   \begin{array}{cc}
     a & b \\
     0 & a^{-1} \\
   \end{array}
 \right))=\left(
           \begin{array}{cc}
             1 & \mu ab+\frac{\lambda}{2}(a^{2}-1) \\
             0 & 1 \\
           \end{array}
         \right), ~~~~\forall \mu, \lambda\in\mathbb{R}\Big\},
\end{eqnarray*}
and
$$
\Gamma_3=\Big\{D: G\lon G~~\big|~~D(\left(
   \begin{array}{cc}
     a & b \\
     0 & a^{-1} \\
   \end{array}
 \right))=\left(
           \begin{array}{cc}
             a^p &  q\frac{a^{2+p}-a^{-p}}{2(p+1)}-ba^{1+p}\\
             0 & a^{-p} \\
           \end{array}
         \right),~~\forall p, q\in\mathbb{R}~~\text{and}~~p\neq 0, -1\Big\}.
$$
\begin{pro}\label{formular}
Let $D$ be a crossed homomorphism with respect to the conjugation action on $G$ and $\huaD: G\times\mathbb{R}\lon G$ be a deformation of $D$. Then there exists $\epsilon>0$ such that $\huaD(\cdot, s)$ and $D$ belong into the same set $\Gamma_i$ for all $s\in (-\epsilon, \epsilon)$.
\end{pro}
\begin{proof}
Suppose that $D\in\Gamma_1$, i.e. $D(\left(
  \begin{array}{cc}
    a & b \\
    0 & a^{-1} \\
  \end{array}
\right))=\left(
          \begin{array}{cc}
            a^{-1} & -b+q_0a\ln a \\
            0 & a \\
          \end{array}
        \right)$ and $\huaD: G\times\mathbb{R}\lon G$ is a deformation of $D$. Then for each $s\in\mathbb{R}$, $\huaD(\cdot, s)$ is a crossed homomorphism which means that $\huaD(\cdot, s)\in \Gamma_1$ or $\Gamma_2$ or $\Gamma_3$.

        Fix $\left(
       \begin{array}{cc}
         a_0 & b_0 \\
         0 & a_0^{-1} \\
       \end{array}
     \right)\in G$, where $a_0\neq 1, b_0\in\mathbb{R}$. As $\huaD(\left(
                                                                      \begin{array}{cc}
                                                                        a_0, & b_0 \\
                                                                        0 & a_0^{-1} \\
                                                                      \end{array}
                                                                    \right)
     , 0)=\left(
            \begin{array}{cc}
              a_0^{-1} & -b_0+q_0a_0\ln a_0 \\
              0 & a_0 \\
            \end{array}
          \right)
     $, there exists an open set $U\ni \huaD(\left(
                                                                      \begin{array}{cc}
                                                                        a_0, & b_0 \\
                                                                        0 & a_0^{-1} \\
                                                                      \end{array}
                                                                    \right)
     , 0)$ of $G$ such that $$U\cap \Big\{\left(
                                      \begin{array}{cc}
                                        1 & \mu a_0b_0+\frac{\lambda}{2}(a_0^{2}-1) \\
                                        0 & 1 \\
                                      \end{array}
                                    \right)\Big|~~\forall \mu, \lambda\in\mathbb{R}\Big\}=\emptyset,
          $$ and $$U\cap\Big\{ \left(
                                                              \begin{array}{cc}
                                                                a_0^{p} & q\frac{a_{0}^{2+p}-a_{0}^{-p}}{2(p+1)}-b_0a_{0}^{1+p} \\
                                                                0 & a_0^{-p} \\
                                                              \end{array}
                                                            \right)\Big|~~\forall p, q\in\mathbb{R}~~\text{and}~~p\neq 0, -1\Big\}=\emptyset.$$Thus there exists $\epsilon>0$ such that
                                                            $\left(
\begin{array}{cc}
a_0 & b_0\\
0& a_0^{-1} \\
\end{array}
\right)\times (-\epsilon, \epsilon)\in\huaD^{-1}(U)$, which means that for any $s\in (-\epsilon, \epsilon)$, we obtain
$$
\huaD(\left(
\begin{array}{cc}
a_0 & b_0\\
0& a_0^{-1} \\
\end{array}
\right), s)\in\Big\{\left(
          \begin{array}{cc}
            a_0^{-1} & -b_0+qa_0\ln a_0 \\
            0 & a_0 \\
          \end{array}
        \right)\Big|~~\forall q\in\mathbb{R}\Big\}.
$$
Therefore, $\huaD(\cdot, s)\in\Gamma_1$ for all $s\in (-\epsilon, \epsilon)$, which means that there exists a smooth function $f:(-\epsilon, \epsilon)\lon \mathbb{R}$ such that
\begin{equation*}
\huaD(\left(
        \begin{array}{cc}
          a & b \\
          0 & a^{-1} \\
        \end{array}
      \right), s
)=\left(
    \begin{array}{cc}
      a^{-1} & -b+f(s)a\ln a \\
      0 & a \\
    \end{array}
  \right), \quad \forall \left(
        \begin{array}{cc}
          a & b \\
          0 & a^{-1} \\
        \end{array}
      \right)\in G, ~~s\in (-\epsilon, \epsilon).
\end{equation*}
Similarly, suppose that $\tilde{D}\in\Gamma_2$ and $\tilde{\huaD}$ is a deformation $\tilde{D}$. There exists $\epsilon_2>0$ and smooth functions
$k:(-\epsilon_2, \epsilon_2)\lon\mathbb{R}, ~~m: (-\epsilon_2, \epsilon_2)\lon\mathbb{R}$ such that
\begin{equation*}
\tilde{\huaD}(\left(
   \begin{array}{cc}
     a & b \\
     0 & a^{-1} \\
   \end{array}
 \right), s)=\left(
           \begin{array}{cc}
             1 & k(s) ab+\frac{m(s)}{2}(a^{2}-1) \\
             0 & 1 \\
           \end{array}
         \right),\quad \forall \left(
        \begin{array}{cc}
          a & b \\
          0 & a^{-1} \\
        \end{array}
      \right)\in G, ~~s\in (-\epsilon_2, \epsilon_2).
\end{equation*}
Similarly, suppose that $\bar{D}\in\Gamma_3$ and $\bar{\huaD}$ is a deformation of $\bar{D}$. There exists $\epsilon_3>0$ and smooth functions $g:(-\epsilon_3, \epsilon_3)\lon\mathbb{R}, ~~h:(-\epsilon_3, \epsilon_3)\lon\mathbb{R}$ where $g(s)\neq0,-1 $ for all $s\in (-\epsilon_3, \epsilon_3)$, such that
$$
\bar{\huaD}(\left(
   \begin{array}{cc}
     a & b \\
     0 & a^{-1} \\
   \end{array}
 \right), s)=\left(
           \begin{array}{cc}
             a^{g(s)} &  h(s)\frac{a^{2+g(s)}-a^{-g(s)}}{2(g(s)+1)}-ba^{1+g(s)}\\
             0 & a^{-g(s)} \\
           \end{array}
         \right), \quad \forall \left(
        \begin{array}{cc}
          a & b \\
          0 & a^{-1} \\
        \end{array}
      \right)\in G, ~~s\in (-\epsilon_3, \epsilon_3).
$$
The proof is finished.
\end{proof}
{\bf Case $1$.} Consider a crossed homomorphism $D$ with respect to the conjugation action on $G$ as following
$$
D(\left(
    \begin{array}{cc}
      a & b \\
      0 & a^{-1} \\
    \end{array}
  \right))=\left(
          \begin{array}{cc}
            a^{-1} & -b+qa\ln a \\
            0 & a \\
          \end{array}
        \right),\quad \forall\left(
    \begin{array}{cc}
      a & b \\
      0 & a^{-1} \\
    \end{array}
  \right)\in G,
$$
where $q$ is a fixed real number. Let $\huaD: G\times\mathbb{R}$ be a deformation of $D$. According to the Proposition \ref{formular}, there exists an interval $(-\epsilon, \epsilon)$ such that
$$
\huaD(\left(
    \begin{array}{cc}
      a & b \\
      0 & a^{-1} \\
    \end{array}
  \right), s)=\left(
          \begin{array}{cc}
            a^{-1} & -b+f(s)a\ln a \\
            0 & a \\
          \end{array}
        \right),\quad \forall\left(
    \begin{array}{cc}
      a & b \\
      0 & a^{-1} \\
    \end{array}
  \right)\in G,~~ \forall s\in(-\epsilon, \epsilon),
$$
where $f: (-\epsilon, \epsilon)\lon\mathbb{R}$ is a smooth function and $f(0)=q$. We always assume that $f(s)$ is not identically zero on any small open interval $I\subset(-\epsilon, \epsilon)$.
\begin{thm}\label{def1}
With the above notations, if $q=f(0)\neq 0$, then the deformation $\huaD$ is trivial.If $q=f(0)=0$, then the deformation $\huaD$ is nontrivial.
\end{thm}
\begin{proof}
By \eqref{eqcoD}, we have that
\begin{eqnarray*}
\hat{D}_s(\left(
           \begin{array}{cc}
             a & b \\
             0 & a^{-1} \\
           \end{array}
         \right))&=&\frac{d}{dt}|_{t=0}\huaD(\left(
                                             \begin{array}{cc}
                                               a & b \\
                                               0 & a^{-1} \\
                                             \end{array}
                                           \right)
         , s+t)\left(
               \begin{array}{cc}
                 a^{-1} & b-f(s)a\ln a \\
                 0 & a \\
               \end{array}
             \right)\\
             &=&\frac{d}{dt}|_{t=0}\left(
                                             \begin{array}{cc}
                                               a^{-1} & -b+f(s+t)a\ln a \\
                                               0 & a \\
                                             \end{array}
                                           \right)\left(
               \begin{array}{cc}
                 a & b-f(s)a\ln a \\
                 0 & a^{-1} \\
               \end{array}
             \right)\\
             &=&\left(
                  \begin{array}{cc}
                    0 & f'(s)\ln a \\
                    0 & 0 \\
                  \end{array}
                \right).
\end{eqnarray*}
If $f(0)\neq0$, there exists an open interval $I\subset(-\epsilon, \epsilon)$ such that $f(s)\neq0$ for all $s\in I$. Define a smooth map $\kappa:I\lon\g$ by $\kappa(s)=\left(
                                                                                                                                                                          \begin{array}{cc}
                                                                                                                                                                            -\frac{f'(s)}{2f(s)} & \kappa_2(s) \\
                                                                                                                                                                             0& \frac{f'(s)}{2f(s)} \\
                                                                                                                                                                          \end{array}
                                                                                                                                                                        \right)
$, where $\kappa_2: I\lon \mathbb{R}$ is a smooth function. Then by \eqref{deficob}, it follows that
\begin{eqnarray*}
\dM^{D_s}(\left(
          \begin{array}{cc}
            -\frac{f'(s)}{2f(s)} & \kappa_2(s) \\
            0 & \frac{f'(s)}{2f(s)} \\
          \end{array}
        \right))\left(
           \begin{array}{cc}
             a & b \\
             0 & a^{-1} \\
           \end{array}
         \right)
         &=&\Theta_{D_s}(\left(
           \begin{array}{cc}
             a & b \\
             0 & a^{-1} \\
           \end{array}
         \right))\left(
          \begin{array}{cc}
            -\frac{f'(s)}{2f(s)} & \kappa_2(s) \\
            0 & \frac{f'(s)}{2f(s)} \\
          \end{array}
        \right)-\left(
           \begin{array}{cc}
            -\frac{f'(s)}{2f(s)}  & \kappa_2(s) \\
             0 & \frac{f'(s)}{2f(s)} \\
           \end{array}
         \right)\\
         &=&\left(
                  \begin{array}{cc}
                    0 & f'(s)\ln a \\
                    0 & 0 \\
                  \end{array}
                \right)=\hat{D}_s\left(
           \begin{array}{cc}
             a & b \\
             0 & a^{-1} \\
           \end{array}
         \right).
\end{eqnarray*}
Moreover, define a smooth map $\tau:I\lon G$ by $
\tau(s)=\left(
          \begin{array}{cc}
           \sqrt{\frac{f(s)}{f(0)}} & \tau_2(s) \\
            0 & \sqrt{\frac{f(0)}{f(s)}} \\
          \end{array}
        \right)
$, where $\tau_2: I\lon\mathbb{R}$ is a smooth function. Then
\begin{eqnarray*}
&&\huaD(\left(
          \begin{array}{cc}
            a & b \\
            0 & a^{-1} \\
          \end{array}
        \right)
,s)\\
&=&\left(
    \begin{array}{cc}
      a^{-1} & -b-f(s)a\ln a \\
      0 & a \\
    \end{array}
  \right)\\&=&\left(
          \begin{array}{cc}
            \sqrt{\frac{f(s)}{f(0)}} & \tau_2(s) \\
            0 & \sqrt{\frac{f(0)}{f(s)}} \\
          \end{array}
        \right)\left(
                 \begin{array}{cc}
                   a^{-1} & -b-qa\ln a \\
                   0 & a \\
                 \end{array}
               \right)\left(
                        \begin{array}{cc}
                          a & b \\
                          0 & a^{-1} \\
                        \end{array}
                      \right)\left(
                               \begin{array}{cc}
                                \sqrt{\frac{f(0)}{f(s)}} & -\tau_2(s) \\
                                 0 & \sqrt{\frac{f(s)}{f(0)}} \\
                               \end{array}
                             \right)\left(
                                      \begin{array}{cc}
                                        a^{-1} & -b \\
                                        0 & a \\
                                      \end{array}
                                    \right)\\
&=&\tau(s)D(\left(
              \begin{array}{cc}
                a & b \\
                0 & a^{-1} \\
              \end{array}
            \right))\left(
                             \begin{array}{cc}
                               a & b \\
                               0 & a^{-1} \\
                             \end{array}
                           \right
            )\tau(s)^{-1}\left(
                             \begin{array}{cc}
                               a^{-1} & -b \\
                               0 & a \\
                             \end{array}
                           \right ).
\end{eqnarray*}
Thus the deformation $\huaD$ is trivial.

If $f(0)=0$, assume that there exists an open interval $I\subset(-\epsilon, \epsilon)$ and a smooth map $\kappa:I\lon\g$ such that $\dM^{D_s}(\kappa(s))=\hat{D}_s$ for all $s\in I$. By \eqref{deficob}, it follows that
\begin{eqnarray*}
\dM^{D_s}(\left(
          \begin{array}{cc}
            \kappa_1(s) & \kappa_2(s) \\
            0 & -\kappa_1(s) \\
          \end{array}
        \right))\left(
           \begin{array}{cc}
             a & b \\
             0 & a^{-1} \\
           \end{array}
         \right)&=&\Theta_{D_s}(\left(
           \begin{array}{cc}
             a & b \\
             0 & a^{-1} \\
           \end{array}
         \right))\left(
          \begin{array}{cc}
            \kappa_1(s) & \kappa_2(s) \\
            0 & -\kappa_1(s) \\
          \end{array}
        \right)-\left(
           \begin{array}{cc}
            \kappa_1(s)  & \kappa_2(s) \\
             0 & -\kappa_1(s) \\
           \end{array}
         \right)\\
         &=&\left(
              \begin{array}{cc}
                0 & -2\kappa_1(s)f(s)\ln a \\
                0 & 0 \\
              \end{array}
            \right).
\end{eqnarray*}
Thus $f'(s)+2\kappa_1(s)f(s)=0$, it means that
\begin{equation}\label{odef}
f(s)e^{2\int_0^{s}\kappa_1(u)du}=f(0)=0.
\end{equation}
As $f(s)$ is not identically zero on any small open interval $J\subset(-\epsilon, \epsilon)$, then \eqref{odef} doesn't hold. Thus we can't find the $I$ and a smooth function $\kappa: I\lon\g$ such that $\dM^{D_s}(\kappa(s))=\hat{D}_s$, which means that the deformation $\huaD$ is nontrivial.
\end{proof}

{\bf Case $2$.}
Consider a crossed homomorphism $D$ with respect to the conjugation action on $G$ as following
$$
D(\left(
    \begin{array}{cc}
      a & b \\
      0 & a^{-1} \\
    \end{array}
  \right))=\left(
          \begin{array}{cc}
            1 & \mu ab+\frac{\lambda}{2}(a^{2}-1) \\
            0 & 1 \\
          \end{array}
        \right),\quad \forall\left(
    \begin{array}{cc}
      a & b \\
      0 & a^{-1} \\
    \end{array}
  \right)\in G,
$$
where $\mu$ and $\lambda$ are fixed real numbers. Let $\huaD: G\times\mathbb{R}\lon G$ be a deformation of $D$. According to the Proposition \ref{formular}, there exists an interval $(-\epsilon, \epsilon)$ such that
$$
\huaD(\left(
    \begin{array}{cc}
      a & b \\
      0 & a^{-1} \\
    \end{array}
  \right), s)=\left(
          \begin{array}{cc}
            1 & k(s) ab+\frac{m(s)}{2}(a^2-1) \\
            0 & 1 \\
          \end{array}
        \right),\quad \forall\left(
    \begin{array}{cc}
      a & b \\
      0 & a^{-1} \\
    \end{array}
  \right)\in G,~~ \forall s\in(-\epsilon, \epsilon),
$$
where $k: (-\epsilon, \epsilon)\lon\mathbb{R}$ and $m: (-\epsilon, \epsilon)\lon\mathbb{R}$ are smooth functions and $k(0)=\mu, m(0)=\lambda$.
\begin{thm}\label{def3}
With the above notations, if $k(0)\neq-1$, then the deformation $\huaD$ is trivial. If $k(0)=-1$ and $k$ is not flat at $0$, which means that there exists $l\in Z_{+}$ such that $k^{(n)}(0)=0$ for all $1\leq n\leq l$, and $k^{(l+1)}(0)\neq 0$ (Here, $k^{(n)}$ means the $n$-th derivative of $k$), then the deformation $\huaD$ is nontrivial.
\end{thm}
\begin{proof}
By \eqref{eqcoD}, we have that
\begin{eqnarray*}
\hat{D}_s(\left(
           \begin{array}{cc}
             a & b \\
             0 & a^{-1} \\
           \end{array}
         \right))&=&\frac{d}{dt}|_{t=0}\huaD(\left(
                                             \begin{array}{cc}
                                               a & b \\
                                               0 & a^{-1} \\
                                             \end{array}
                                           \right)
         , t+s)\left(
               \begin{array}{cc}
                 1 & -k(s) ab-\frac{m(s)}{2}(a^2-1) \\
                 0 & 1 \\
               \end{array}
             \right)\\
             &=&\frac{d}{dt}|_{t=0}\left(
               \begin{array}{cc}
                 1 & k(s+t)ab+\frac{m(s+t)}{2}(a^2-1)-k(s)ab-\frac{m(s)}{2}(a^{2}-1) \\
                 0 & 1 \\
               \end{array}
             \right)\\
             &=&\left(
                  \begin{array}{cc}
                    0 & k'(s)ab+\frac{m'(s)}{2}(a^{2}-1) \\
                    0 & 0 \\
                  \end{array}
                \right).
\end{eqnarray*}
If $k(0)\neq-1$, there exists an open interval $I\subset(-\epsilon, \epsilon)$ such that $k(s)\neq-1$ for all $s\in I$. Let $\alpha_1(s)=-\frac{k'(s)}{2(k(s)+1)}$ and $\alpha_2(s)=\frac{m'(s)}{2}-\frac{m(s)k'(s)}{2(k(s)+1)}$. Then by \eqref{deficob}, it follows that
\begin{eqnarray*}
\dM^{D_s}(\left(
          \begin{array}{cc}
           \alpha_1(s) & \alpha_2(s) \\
            0 & -\alpha_1(s) \\
          \end{array}
        \right))\left(
           \begin{array}{cc}
             a & b \\
             0 & a^{-1} \\
           \end{array}
         \right)
&=&\Theta_{D_s}(\left(
           \begin{array}{cc}
             a & b \\
             0 & a^{-1} \\
           \end{array}
         \right))\left(
          \begin{array}{cc}
            \alpha_1(s) & \alpha_2(s) \\
            0 & -\alpha_1(s) \\
          \end{array}
        \right)-\left(
           \begin{array}{cc}
            \alpha_1(s)  & \alpha_2(s) \\
             0 & -\alpha_1(s) \\
           \end{array}
         \right)\\
         &=&\left(
              \begin{array}{cc}
                0 & k'(s)ab+\frac{m'(s)}{2}(a^{2}-1) \\
                0 & 0 \\
              \end{array}
            \right)=\hat{D}_s(\left(
           \begin{array}{cc}
             a & b \\
             0 & a^{-1} \\
           \end{array}
         \right)).
\end{eqnarray*}
Moreover, define a smooth map $\tau:I\lon G$ by $
\tau(s)=\left(
          \begin{array}{cc}
           \sqrt{\frac{k(s)+1}{k(0)+1}} & \frac{-m(s)(k(0)+1)+m(0)(k(s)+1)}{2\sqrt{(k(0)+1)(k(s)+1)}} \\
            0 & \sqrt{\frac{k(0)+1}{k(s)+1}} \\
          \end{array}
        \right)
$.
Then
\begin{eqnarray*}
&&\huaD(\left(
          \begin{array}{cc}
            a & b \\
            0 & a^{-1} \\
          \end{array}
        \right)
,s)
=\tau(s)D(\left(
              \begin{array}{cc}
                a & b \\
                0 & a^{-1} \\
              \end{array}
            \right))\left(
                             \begin{array}{cc}
                               a & b \\
                               0 & a^{-1} \\
                             \end{array}
                           \right)\tau(s)^{-1}\left(
                             \begin{array}{cc}
                               a^{-1} & -b \\
                               0 & a \\
                             \end{array}
                           \right).
\end{eqnarray*}
Thus the deformation $\huaD$ is trivial.

If $k(0)=-1$ and $k$ is not flat at $0$, then there exists an open interval $0\in J\subset(-\epsilon, \epsilon)$ such that $$k(s)+1=s^{l+1}h(s)$$ for all $s\in J$, where $h(0)\neq 0$. Assume that there exists an open interval $I\subset J$ and a smooth map $\alpha:I\lon\g$ such that $\dM^{D_s}(\alpha(s))=\hat{D}_s$ for all $s\in I$. By \eqref{deficob}, it follows that
\begin{eqnarray*}
&&\dM^{D_s}(\left(
          \begin{array}{cc}
           \alpha_1(s) & \alpha_2(s) \\
            0 & -\alpha_1(s) \\
          \end{array}
        \right))\left(
           \begin{array}{cc}
             a & b \\
             0 & a^{-1} \\
           \end{array}
         \right)\\
&=&\left(
                   \begin{array}{cc}
                     0 & -2(k(s)\alpha_1(s)+\alpha_1(s))ab+(\alpha_2(s)-m(s)\alpha_1(s))(a^{2}-1) \\
                     0 & 0 \\
                   \end{array}
                 \right).
\end{eqnarray*}
Thus
\begin{eqnarray}
\label{ode1}2(k(s)+1)\alpha_1(s)&=&-k'(s),\\
\nonumber\frac{m'(s)}{2}+m(s)\alpha_1(s)&=&\alpha_2(s),
\end{eqnarray}
which means that $$2sh(s)\alpha_1(s)=-(l+1)h(s)-sh'(s).$$
We have that $0=(l+1)h(0)\neq 0$. Therefore we can't find the $I$ and a smooth function $\alpha: I\lon\g$ such that $\dM^{D_s}(\alpha(s))=\hat{D}_s$, which means that the deformation $\huaD$ is nontrivial.
\end{proof}
\begin{rmk}
When a smooth function $k(s)$ is flat at $s=0$ and $k(0)=-1$. Then it is not possible to determine whether the deformation is trivial or not. For examples, if
\begin{equation*}
k(s)=
\begin{cases}
-1+e^{-\frac{1}{s^{2}}}, & s\neq0;\\
-1, & s=0,
\end{cases}
\end{equation*}
 assume that there exists $\alpha_1(s)$ satisfies the \eqref{ode1}, then $\alpha_1(s)=-\frac{1}{s^3}$ for $s\neq 0$. But $\displaystyle\lim_{s\rightarrow 0}\alpha_1(s)=\infty$, we can't find a smooth function $\alpha_1(s)$ defined on $(-\epsilon, \epsilon)$ satisfies \eqref{ode1}. Thus $\huaD$ is nontrivial.

If $k(s)=-1$ for all $s\in\mathbb{R}$, Let $\tau(s)=\left(
                                                      \begin{array}{cc}
                                                        1 & \frac{\lambda-m(s)}{2} \\
                                                        0 & 1 \\
                                                      \end{array}
                                                    \right)
$. We have
\begin{eqnarray*}
\huaD(\left(
        \begin{array}{cc}
          a & b \\
          0 & a^{-1} \\
        \end{array}
      \right)
,s)
&=&\left(
          \begin{array}{cc}
            1 & -ab+\frac{m(s)}{2}(a^2-1) \\
            0 & 1 \\
          \end{array}
        \right)\\
        &=&\tau(s)\left(
                          \begin{array}{cc}
                            1 & -ab+\frac{\lambda}{2}(a^{2}-1) \\
                            0 & 1 \\
                          \end{array}
                        \right)\left(
                                        \begin{array}{cc}
                                          a & b \\
                                          0 & a^{-1} \\
                                        \end{array}
                                      \right)
                        \tau(s)^{-1}\left(
                             \begin{array}{cc}
                               a^{-1} & -b \\
                               0 & a \\
                             \end{array}
                           \right),
\end{eqnarray*}
which means that $\huaD$ is trivial.
\end{rmk}

{\bf Case $3$.}
Consider a crossed homomorphism $D$ with respect to the conjugation action on $G$ as following
$$
D(\left(
    \begin{array}{cc}
      a & b \\
      0 & a^{-1} \\
    \end{array}
  \right))=\left(
           \begin{array}{cc}
             a^{p} &  q\frac{a^{2+p}-a^{-p}}{2(p+1)}-ba^{1+p}\\
             0 & a^{-p} \\
           \end{array}
         \right),\quad \forall\left(
    \begin{array}{cc}
      a & b \\
      0 & a^{-1} \\
    \end{array}
  \right)\in G,
$$
where $p$ and $q$ are fixed real numbers and $p\neq 0, -1$. Let $\huaD: G\times\mathbb{R}$ be a deformation of $D$. According to the Proposition \ref{formular}, there exists an interval $(-\epsilon, \epsilon)$ such that
$$
\huaD(\left(
    \begin{array}{cc}
      a & b \\
      0 & a^{-1} \\
    \end{array}
  \right), s)=\left(
           \begin{array}{cc}
             a^{g(s)} &  h(s)\frac{a^{2+g(s)}-a^{-g(s)}}{2(g(s)+1)}-ba^{1+g(s)}\\
             0 & a^{-g(s)} \\
           \end{array}
         \right), \quad \forall\left(
    \begin{array}{cc}
      a & b \\
      0 & a^{-1} \\
    \end{array}
  \right)\in G,~~ \forall s\in(-\epsilon, \epsilon),
$$
where $h: (-\epsilon, \epsilon)\lon\mathbb{R}$ and $g: (-\epsilon, \epsilon)\lon\mathbb{R}$ are smooth functions satisfying $h(0)=q, g(0)=p$ and $g(s)\neq0, -1$ for all $s\in (-\epsilon, \epsilon)$.
\begin{thm}\label{def2}
With the above notations, if there exists an open interval $I\subset (-\epsilon, \epsilon)$ such that $g'(s)=0$ for all $s\in I$, then the deformation $\huaD$ is trivial. Otherwise, the deformation $\huaD$ is nontrivial.
\end{thm}
\begin{proof}
By \eqref{eqcoD}, we have that
\begin{eqnarray*}
&&\hat{D}_s(\left(
           \begin{array}{cc}
             a & b \\
             0 & a^{-1} \\
           \end{array}
         \right))\\
         &=&\frac{d}{dt}|_{t=0}\huaD(\left(
                                             \begin{array}{cc}
                                               a & b \\
                                               0 & a^{-1} \\
                                             \end{array}
                                           \right)
         , s+t)\left(
                                             \begin{array}{cc}
                                               a^{-g(s)} & -h(s)\frac{a^{2+g(s)}-a^{-g(s)}}{2(g(s)+1)}+ba^{1+g(s)} \\
                                               0 & a^{g(s)} \\
                                             \end{array}
                                           \right)\\
             &=&\frac{d}{dt}|_{t=0}\left(
                                             \begin{array}{cc}
                                               a^{g(s+t)} & h(s+t)\frac{a^{2+g(s+t)}-a^{-g(s+t)}}{2(g(s+t)+1)}-ba^{1+g(s+t)} \\
                                               0 & a^{-g(s+t)} \\
                                             \end{array}
                                           \right)\left(
                                             \begin{array}{cc}
                                               a^{-g(s)} & -h(s)\frac{a^{2+g(s)}-a^{-g(s)}}{2(g(s)+1)}+ba^{1+g(s)} \\
                                               0 & a^{g(s)} \\
                                             \end{array}
                                           \right)\\
             &=&\left(
                  \begin{array}{cc}
                    g'(s)\ln a & \frac{h(s)g'(s)\ln a}{g(s)+1}+\frac{h'(s)(a^{2+2g(s)}-1)}{2g(s)+2}-\frac{2h(s)g'(s)(a^{2+2g(s)}-1)}{(2g(s)+2)^{2}} \\
                    0 & -g'(s)\ln a \\
                  \end{array}
                \right).
\end{eqnarray*}
If there exists an open interval $I\subset (-\epsilon, \epsilon)$ such that $g'(s)=0$ for all $s\in I$, define $$\kappa(s)=\left(
                                                                                                                        \begin{array}{cc}
                                                                                                                          \kappa_1(s) & \frac{h'(s)+2h(s)\kappa_1(s)}{2+2p} \\
                                                                                                                          0 & -\kappa_1(s) \\
                                                                                                                        \end{array}
                                                                                                                      \right),$$ where $\kappa_1: I\lon\mathbb{R}$ is a smooth function. Then
\begin{eqnarray*}
&&\dM^{D_s}(\left(
          \begin{array}{cc}
            \kappa_1(s) & \frac{h'(s)+2h(s)\kappa_1(s)}{2+2p} \\
            0 & -\kappa_1(s) \\
          \end{array}
        \right))\left(
           \begin{array}{cc}
             a & b \\
             0 & a^{-1} \\
           \end{array}
         \right)\\
&=&\Theta_{D_s}(\left(
           \begin{array}{cc}
             a & b \\
             0 & a^{-1} \\
           \end{array}
         \right))\left(
          \begin{array}{cc}
            \kappa_1(s) & \frac{h'(s)+2h(s)\kappa_1(s)}{2+2p} \\
            0 & -\kappa_1(s) \\
          \end{array}
        \right)-\left(
           \begin{array}{cc}
            \kappa_1(s)  & \frac{h'(s)+2h(s)\kappa_1(s)}{2+2p} \\
             0 & -\kappa_1(s) \\
           \end{array}
         \right)=\hat{D}_s(\left(
                             \begin{array}{cc}
                               a & b \\
                               0 & a^{-1} \\
                             \end{array}
                           \right)).
\end{eqnarray*}
Moreover, define a smooth map $\tau:I\lon G$ by $
\tau(s)=\left(
          \begin{array}{cc}
           \tau_1(s) & \frac{q\tau(s)}{2p+2}-\frac{h(s)}{(2p+2)\tau_1(s)} \\
            0 & \tau(s)^{-1} \\
          \end{array}
        \right)
$, where $\tau_1: I\lon \mathbb{R}$ is a smooth function. Then
\begin{eqnarray*}
&&\huaD(\left(
          \begin{array}{cc}
            a & b \\
            0 & a^{-1} \\
          \end{array}
        \right)
,s)
=\tau(s)D(\left(
              \begin{array}{cc}
                a & b \\
                0 & a^{-1} \\
              \end{array}
            \right))\left(
                             \begin{array}{cc}
                               a & b \\
                               0 & a^{-1} \\
                             \end{array}
                           \right)\tau(s)^{-1}\left(
                             \begin{array}{cc}
                               a^{-1} & -b \\
                               0 & a \\
                             \end{array}
                           \right).
\end{eqnarray*}
Thus the deformation $\huaD$ is trivial.

If $g'(s)$ is not identically zero on any small open interval $J\subset(-\epsilon, \epsilon)$, assume that there is an open interval $J_1\subset(-\epsilon, \epsilon)$ and a smooth map $\kappa: J_1\lon\mathbb{R}$ such that
\begin{eqnarray*}
&&\dM^{D_s}(\left(
          \begin{array}{cc}
            \kappa_1(s) & \kappa_2(s) \\
            0 & -\kappa_1(s) \\
          \end{array}
        \right))\left(
           \begin{array}{cc}
             a & b \\
             0 & a^{-1} \\
           \end{array}
         \right)\\
&=&\Theta_{D_s}(\left(
           \begin{array}{cc}
             a & b \\
             0 & a^{-1} \\
           \end{array}
         \right))\left(
          \begin{array}{cc}
            \kappa_1(s) & \kappa_2(s) \\
            0 & -\kappa_1(s) \\
          \end{array}
        \right)-\left(
           \begin{array}{cc}
            \kappa_1(s)  & \kappa_2(s) \\
             0 & -\kappa_1(s) \\
           \end{array}
         \right)\\
         &=&\left(
              \begin{array}{cc}
                0 & \kappa_2(s)(a^{2+2g(s)}-1)-\frac{h(s)\kappa_1(s)(a^{2+2g(s)}-1)}{g(s)+1} \\
                0 & 0 \\
              \end{array}
            \right)
         \\
         &=&\hat{D}_s(\left(
                             \begin{array}{cc}
                               a & b \\
                               0 & a^{-1} \\
                             \end{array}
                           \right)),
\end{eqnarray*}
which means that $g'(s)=0$ for all $s\in J_1$. This contradicts the fact that $g'(s)$ is not identically zero on any small open interval $J\subset(-\epsilon, \epsilon)$. Thus the deformation $\huaD$ is nontrivial.
\end{proof}

{\bf Conclusion.} The rigidity properties of all crossed homomorphisms with respect to the conjugation action on $G$ are characterized as follows:
\begin{itemize}
\item The crossed homomorphisms given by
$$
D(\left(
   \begin{array}{cc}
     a & b \\
     0 & a^{-1} \\
   \end{array}
 \right))=\left(
           \begin{array}{cc}
             a^p &  q\frac{a^{2+p}-a^{-p}}{2(p+1)}-ba^{1+p}\\
             0 & a^{-p} \\
           \end{array}
         \right),~~\forall p, q\in\mathbb{R}~~\text{and}~~p\neq 0, -1,
         $$
are {\bf nonrigid}.
\item The other cases,
\begin{table}[htbp]
  \centering
    \begin{tabular}{|c|c|c|c|}
    \hline
    \multicolumn{1}{|c|}{Crossed homomorphisms} & Cases & Rigid   \bigstrut\\
    \hline
    \multicolumn{1}{|c|}{\multirow{2}{*}{$D\begin{pmatrix} a & b \\ 0 & a^{-1} \end{pmatrix}=\begin{pmatrix} a^{-1} & -b+qa\ln a \\ 0 & a \end{pmatrix}, \quad q\in\mathbb{R}$}} & $q=0$ & {\bf No}  \bigstrut\\
\cline{2-3}          & $q\neq 0$ & {\bf Yes}     \bigstrut\\
\cline{1-3}    \multicolumn{1}{|c|}{\multirow{2}{*}{$D\begin{pmatrix} a & b \\ 0 & a^{-1} \end{pmatrix}=\begin{pmatrix} 1 & \mu ab+\frac{\lambda}{2}(a^2-1) \\ 0 & 1 \end{pmatrix}, \quad \lambda, \mu\in\mathbb{R}$}} & $\mu=-1$ & {\bf No}  \bigstrut\\
\cline{2-3}          & $\mu\neq-1$ & {\bf Yes}  \bigstrut\\
    \hline
    \end{tabular}~~.
\end{table}
\end{itemize}

 \end{document}